\documentclass[11pt,thmsa]{article}%
\usepackage{amsmath}
\usepackage{amsfonts}
\usepackage{amssymb}
\usepackage{graphicx}%
\newtheorem{theorem}{Theorem}[section]

\newtheorem{corollary}[theorem]{Corollary}

\newtheorem{lemma}[theorem]{Lemma}

\newtheorem{proposition}[theorem]{Proposition}

\newenvironment{proof}[1][Proof]{\noindent\textbf{#1.} }{\ \rule{0.5em}{0.5em}}

\begin{document}
\title{Geodesic orbit spheres and constant curvature in Finsler geometry
\thanks{Supported by NSFC (no. 11771331) and NSF of Beijing (no. 1182006)}}
\author{Ming Xu\\
\\
School of Mathematical Sciences\\
Capital Normal University\\
Beijing 100048, P. R. China\\
Email: mgmgmgxu@163.com
\\
}
\date{}
\maketitle

\begin{abstract}
In this paper, we generalize the classification of geodesic orbit spheres from Riemannian geometry to Finsler geometry. Then we further prove if a geodesic orbit Finsler sphere has constant flag curvature,
it must be Randers. It provides an alternative proof for the classification
of invariant Finsler metrics with $K\equiv1$ on homogeneous spheres other
than $Sp(n)/Sp(n-1)$.



\end{abstract}

\section{Introduction}
A Riemannian homogeneous manifold is called a {\it geodesic orbit space}, if any geodesic is the orbit of a one-parameter subgroup of isometries.
This notion was introduced by
by O. Kowalski and L. Vanhecke in 1991 \cite{KV1991}, as a
generalization of naturally reductive homogeneity.
Since then, there have been many research works on this subject. See \cite{AA2007}\cite{AN2009}\cite{AV1999}\cite{BN2018}\cite{DKN}\cite{Go1996}\cite{GN2018}
for example.

Meanwhile, geodesic orbit metrics have also been studied in Finsler geometry. In \cite{YD2014}, the notion of geodesic orbit Finsler space was defined, and in \cite{XD2017}, the
interaction between geodesic orbit property and negative curvature property was explored.

The first purpose of the paper is to generalize Yu.G. Nikonorov's classification of geodesic orbit metrics on spheres
\cite{Ni2013} to Finsler geometry, and prove the following theorem.

\begin{theorem}\label{main-thm-1}
A homogeneous Finsler metric $F$ on a sphere $M=S^n$ with $n>1$ is a geodesic orbit metric
iff the connected isometry group $I_o(M,F)$ is not isomorphic to $Sp(k)$ for any $k\geq 1$.
\end{theorem}

By this theorem, we can easily
list all the geodesics orbit metrics on spheres:
\begin{description}
\item{\rm (1)} Riemannian metrics of constant curvature.
\item{\rm (2)} All homogeneous Finsler metrics on
$S^{4n-1}=Sp(n)Sp(1)/Sp(n-1)Sp(1)$ with $n>1$ and $S^{15}=Spin(9)/Spin(7)$. They are all of $(\alpha_1,\alpha_2)$-type, in which some special ones are Riemannian.
\item{\rm (3)} All homogeneous Finsler metrics on $SU(n)/SU(n-1)$ with $n>2$ and $U(n)/U(n-1)$ with
    $n>1$. They are all of $(\alpha,\beta)$-type, in which some special ones are Riemannian.
\item{\rm (4)} All homogeneous Finsler metrics on
$Sp(n)U(1)/Sp(n-1)U(1)$. They are all of $(\alpha_1,\alpha_2,\beta)$-type, in which some special ones are of $(\alpha,\beta)$-type, $(\alpha_1,\alpha_2)$-type, or Riemannian.
\end{description}

See Section 2 for the notions of these metrics. When the metrics are Riemannian,
the above list re-verifies Table 1 in \cite{Ni2013}. For each case of (2)-(4),
the space of geodesic orbit metrics has an infinite dimension.
In an independent work \cite{ZD2018-preprint},
S. Zhang and S. Deng classified geodesic orbit Randers spheres with a more algebraic method, and described their geodesic vectors.

The second purpose of this paper is to apply Theorem \ref{main-thm-1} to
homogeneous Finsler spheres of constant flag curvature $K\equiv 1$, and
explore the interaction between geodesic orbit property and constant positive curvature property. We will prove the following theorem.

\begin{theorem}\label{main-thm-2}
A homogeneous Finsler sphere $(M,F)=(S^n,F)$ with $n>1$ and $K\equiv 1$ is a geodesic orbit
space iff it is Randers.
\end{theorem}

All Randers spheres $(M,F)=(S^n,F)$ with $n>1$ and $K\equiv 1$ are classified by D. Bao, Z. Shen and C. Robles \cite{BRS2004}, i.e. the metric $F$ must be defined by the navigation datum $(h,W)$, in which $h$ is the
Riemannian metric for the unit sphere, and $W$ is a Killing vector field with $h(W,W)<1$ everywhere. The only new ingredient is that
$F$ is homogeneous iff $W$ has a constant $h$-length.
So we have the following corollary of Theorem \ref{main-thm-1} and Theorem \ref{main-thm-2},
\begin{corollary}\label{cor-classification}
Any invariant Finsler metric $F$ on a homogeneous sphere $M$ with $\dim M>1$, $K\equiv 1$ and $I_o(M,F)\neq Sp(k)$ for all $k\in\mathbb{N}$, i.e. $M\neq Sp(n)/Sp(n-1)$ for all $n>1$
or $Sp(1)/Sp(0)=SU(2)/\{e\}$ in the list
(\ref{list-homo-sphere}) of homogeneous spheres, then $F$ is a Randers metric defined by the navigation
datum $(h,W)$ in which $h$ is the Riemannian metric for the unit sphere and $W$ is a Killing vector field
of constant length on $(M,h)$.
\end{corollary}

Corollary \ref{cor-classification} has classified all the invariant Finsler
metrics with constant flag curvature on a homogeneous sphere, except for the most difficult case $M=Sp(n)/Sp(n-1)$. It implies a negative answer to the question if a homogeneous Finsler metric of constant flag curvature can be "exotic". Global homogeneity for the metric is a critical condition, because in the non-homogeneous situation, we know the examples discovered by R.L. Bryant
\cite{Br1996}\cite{Br1997}\cite{Br2002}, and there may exist many more.

The above theorems and corollary can also be applied to study a homogeneous Finsler sphere $(M,F)$ with
$K\equiv1$ and finite orbits of prime closed geodesics \cite{Xu2018-preprint}. By Theorem \ref{main-thm-2} in this paper and that in \cite{Xu2018-preprint}, we may find totally geodesic sub-manifolds
of $M$ which are Randers spheres.

By private communication, the author noticed that L. Huang had discovered Corollary \ref{cor-classification} in 2015, and found a computational proof
based on his homogeneous flag curvature formula \cite{Hu2015}\cite{Hu2017}. The method in this paper
is more geometrical, and has not used any calculation concerning L. Huang formula.
A Killing navigation process has been applied to reduce our discussion
to the case that $(M,F)$ is a geodesic orbit sphere with $K\equiv1$, all geodesics are closed, and all
prime closed geodesics have the same length $2\pi$. Then we use the geodesic orbit property to prove
the reversibility, and finally we apply a theorem of C. Kim and K. Min, \cite{KM2009} (see Theorem \ref{reversible-constant-flag-curvature}) to finish the proof.

This paper is organized as following. In Section 2, we summarize some basic knowledge in Finsler geometry
and introduce the notions for some special classes of metrics. In Section 3, we review the concepts and some fundamental properties of geodesic orbit metrics and weakly symmetric spaces. In Section 4, we prove
Theorem \ref{main-thm-1}. In Section 5, we introduce the technique of Killing navigation process and two propositions for it. In Section 6, we prove Theorem \ref{main-thm-2}.

{\bf Acknowledgement.} The author would like to thank sincerely Chern Institute of Mathematics, Nankai University and Shaoqiang Deng for the hospitality during the preparation for this paper. The author also
thanks Yuri G. Nikonorov, Libing Huang, Shaoqiang Deng, Zhongmin Shen and Huaifu Liu for helpful discussions and suggestions.

\section{Finsler metric and some examples}

A {\it Finsler metric} on a smooth manifold $M$ with $\dim M=n$
is a continuous function $F:TM\rightarrow [0,+\infty)$, satisfying
the following properties \cite{CS2005}:
\begin{description}
\item{\rm (1)} The restriction of $F$ to the slit tangent bundle $TM\backslash
0$ is a positive smooth function.
\item{\rm (2)} For any $\lambda\geq 0$,
$F(x,\lambda y)=\lambda F(x,y)$.
\item{\rm (3)} For any local coordinates $x=(x^i)\in M$
and $y=y^j\partial_{x^j}\in T_xM$, the Hessian matrix
$$(g_{ij}(x,y))=\left(\frac12[F^2(x,y)]_{y^iy^j}\right)$$
is positive definite.
\end{description}

We will also call $(M,F)$ a {\it Finsler manifold} or a {\it Finsler space}.
The restriction of $F$ to each tangent space $T_xM$ is called a {\it Minkowski norm}.
Minkowski norm can be similarly defined on any real linear space.
If a Finsler metric $F$ satisfies $F(x,y)=F(x,-y)$ for any $x\in T_xM$ and
$y\in T_xM$, we call $F$ {\it reversible}.

Here are some examples.

A Finsler metric $F$ is {\it Riemannian}, if for any standard
local coordinates, the Hessian matrix $(g_{ij}(x,y))$ depends only on the $x$-variables.

The most simple and the most important non-Riemannian Finsler metrics are {\it Randers metrics},
which are of the form $F=\alpha+\beta$, in which $\alpha$ is a Riemannian metric and $\beta$ is
a one-form.

The {\it $(\alpha,\beta)$-metrics}
are of the form $F=\alpha\phi(\beta/\alpha)$, in which $\phi$ is a positive smooth function, and
the pair $\alpha$ and $\beta$ are similar to those for Randers metrics.

Applying similar thoughts, we can construct more Finsler metrics.

For example,
the $(\alpha_1,\alpha_2)$-metrics are defined in \cite{DX2016}.
Let $\alpha$ be a Riemannian metric on $M$, and
$TM=\mathcal{V}_1+\mathcal{V}_2$ is an $\alpha$-orthogonal bundle decomposition. Denote $\alpha_i=\alpha|_{\mathcal{V}_i}$.
Then a Finsler metric $F$ is called an {\it $(\alpha_1,\alpha_2)$-metric} if
$F(x,y)=f(\alpha_1(x,y_1),\alpha_2(x,y_2))$ for some fixed function $f$, where $y=y_1+y_2\in T_xM$ with
$y_i\in\mathcal{V}_i$ respectively.

In later discussion, we will meet the following {\it $(\alpha_1,\alpha_2,\beta)$-metrics}, which
have a combined feature from both $(\alpha,\beta)$-metrics and $(\alpha_1,\alpha_2)$-metrics. Let
$\alpha$ be a Riemannian metric and $TM=\mathcal{V}_0+\mathcal{V}_1+\mathcal{V}_2$ an $\alpha$-orthogonal bundle decomposition such that the dimension of the fibers of $\mathcal{V}_0$ are one dimensional. Denote $\beta$ a one-form with $\mathrm{\ker}\beta=\mathcal{V}_1+\mathcal{V}_2$, and $\alpha_i=\alpha|_{\mathcal{V}_i}$ for $i=1$ and 2. Then a Finsler metric $F$ is called an {\it $(\alpha_1,\alpha_2,\beta)$-metric}
if $F(x,y)=f(\beta(x,y_0),\alpha_1(x,y_1),\alpha_2(x,y_2))$ for some fixed function $f$, where $y=y_0+y_1+y_2$
with $y_i\in\mathcal{V}_i$ respectively.

In homogeneous Finsler geometry, we consider a smooth coset space $G/H$ such that
$H$ is a compact subgroup of $G$, and there is a {\it reductive decomposition}
$\mathfrak{g}=\mathfrak{h}+\mathfrak{m}$ which identifies the tangent space at $o=eH$ with the
$\mathrm{Ad}(H)$-invariant subspace $\mathfrak{m}$. Then any $G$-invariant Finsler metric $F$
is one-to-one determined by an $\mathrm{Ad}(H)$-invariant Minkowski norm on $\mathfrak{m}$ \cite{De2012}. Generally speaking, when
$G/H$ is connected and simply connected, the homogeneous metric $F$ is of Randers, $(\alpha,\beta)$-,
$(\alpha_1,\alpha_2)$- or $(\alpha_1,\alpha_2,\beta)$-type, iff its Minkowski norm on $\mathfrak{m}$
belongs to the same type, which can be easily detected by studying the isotropy representations
and the linear isometry group $O(\mathfrak{m},F)$.

\section{Geodesic orbit metric and weakly symmetric space}

A homogeneous Finsler space $(G/H,F)$ is called a $G$-geodesic orbit metric, if any geodesic of positive constant speed is the
orbit of the one-parameter subgroup generated by some Killing
vector field $X\in\mathfrak{g}=\mathrm{Lie}(G)$. If the group
$G$ is not specified, we will assume $G=I_o(M,F)$.

The following proposition provides several equivalent definitions.

\begin{proposition}\label{prop-1}
Let $(G/H,F)$ be a homogeneous Finsler space, with a reductive decomposition
$\mathfrak{g}=\mathfrak{h}+\mathfrak{m}$, and denote $[\cdot,\cdot]_\mathfrak{m}$ the
$\mathfrak{m}$-factor in the bracket operation $[\cdot,\cdot]$. Then the following are equivalent:
\begin{description}
\item{\rm (1)} $F$ is a $G$-geodesic orbit metric.
\item{\rm (2)} For any $x\in M$, and any nonzero $y\in T_xM$, we can find a Killing vector
field $X\in\mathfrak{g}$ such that $X(x)=y$ and $x$ is a critical point for the function $f(\cdot)=F(X(\cdot))$.
\item{\rm (3)} For any nonzero vector $u\in\mathfrak{m}$, there exists an $u'\in\mathfrak{h}$
such that
$$\langle u,[u+u',\mathfrak{m}]_\mathfrak{m}\rangle_u=0.$$
\item{\rm (4)} The spray vector field $\eta(\cdot):\mathfrak{m}\backslash \{0\}\rightarrow\mathfrak{m}$
is tangent to the $\mathrm{Ad}(H)$-orbits.
\end{description}
\end{proposition}

The notion $\langle\cdot,\cdot\rangle_u$ is the inner product
defined by the Hessian $(g_{ij}(o,u))$ of $F$.
The spray vector field $\eta(\cdot)$ was defined by L. Huang
in \cite{Hu2015}, by the following equality,
$$\langle\eta(y),u\rangle_y=\langle y,[u,y]_\mathfrak{m}
\rangle_y,\quad\forall u\in\mathfrak{m}.$$

Now we sketch the proof of Proposition \ref{prop-1}.

The equivalence between (1) and (2) is guaranteed by the following lemma, which follows easily Lemma 3.1
in \cite{DX2014}.

\begin{lemma}\label{lemma-0}
Let $X$ be a Killing vector field on a Finsler space
$(M,F)$ which is non-vanishing at $x\in M$. Then the
integration curve of $X$ passing $x$ is a geodesic
iff $x$ is a critical point for the function $f(\cdot)=F(X(\cdot))$.
\end{lemma}

The equivalence between (1) and (3) is provided by Proposition 3.4
in \cite{YD2014}.

So we only need
to prove the equivalence between (3) and (4). Assuming (3), for
any nonzero vector $u\in\mathfrak{m}$, there exists an $u'\in\mathfrak{h}$,
such that $\langle u,[u+u',\mathfrak{m}]_\mathfrak{m}\rangle_u=0$. Then for any
$v\in\mathfrak{m}$, we have
$$\langle\eta(u),v\rangle_u=\langle u,[v,u]_\mathfrak{m}\rangle_u=
\langle u,[u',v]\rangle_u.$$
By Theorem 5.1 in \cite{DH2004} and a fundamental property of the Cartan tensor,
$$\langle u,[u',v]\rangle_u=\langle [u,u'],v\rangle-
2C_u(u,v,[u',u])=\langle[u,u'],v\rangle.$$
So we have $$\langle\eta(u),v\rangle_u=\langle[u,u'],v\rangle_u,
\quad\forall v\in\mathfrak{m},$$
i.e. $\eta(u)=[u,u']\in[u,\mathfrak{h}]=T_u(\mathrm{Ad}(H)\cdot u)$, which proves (4) from (3).

Reversing the above steps, we can also prove (3)
from (4).

This ends the proof of Proposition \ref{prop-1}.

There are many compact and non-compact examples of geodesic orbit Finsler spaces. For example, see \cite{XD2017-2}\cite{XZ2018} for the normal homogeneous Finsler spaces and $\delta$-homogeneous Finsler spaces, and see \cite{YD2014} for
nilmanifold examples.

A smooth coset space $G/H$ in which $H$ is compact is called
{\it weakly symmetric} if for any vector $y\in\mathfrak{m}$, there exists
$g\in H$, such that $\mathrm{Ad}(g)\cdot y=-y$.
Any invariant
Finsler metric $F$ on the weakly symmetric coset space $G/H$
is weakly symmetric \cite{DH2010}, thus it is reversible and
it is a geodesic orbit metric.

In later discussion, we will only meet the following two examples of weakly symmetric coset spaces,
$S^{4n-1}=Sp(n)Sp(1)/Sp(n-1)Sp(1)$ with $n>1$ and $S^{15}=Spin(9)/Spin(7)$.
For each one of these two coset spaces, $\mathfrak{m}$
can be decomposed as the sum $\mathfrak{m}=\mathfrak{m}_1+\mathfrak{m}_2$
of two different irreducible $\mathrm{Ad}(H)$-representation subspaces.
A generic $\mathrm{Ad}(H)$-orbit in $\mathfrak{m}$ is
of cohomogeneity two, i.e. a product of two spheres
in $\mathfrak{m}_1$ and $\mathfrak{m}_2$ respectively.
So the invariant Finsler metrics on them must be of
$(\alpha_1,\alpha_2)$-type, corresponding to the decomposition
of $\mathfrak{m}$.

See \cite{Ya2004} for more examples and its classification theory in Riemannian geometry.
\section{Proof of Theorem \ref{main-thm-1}}

Now we prove Theorem \ref{main-thm-1}.

Firstly, it is a well known fact that all homogeneous spheres
$G/H$ where $G$ is a compact connected Lie group with an effective action on $G/H$ are given by the following list:

\begin{eqnarray}
& &SO(n)/SO(n-1),\quad SU(n)/SU(n-1),\quad U(n)/U(n-1),\nonumber\\
& &Sp(n)/Sp(n-1),\quad Sp(n)U(1)/Sp(n-1)U(1),\quad Sp(n)Sp(1)/Sp(n-1)Sp(1),\nonumber\\
& &G_2/SU(3),\quad  Spin(7)/G_2,\quad Spin(9)/Spin(7),\label{list-homo-sphere}
\end{eqnarray}
in which $n>1$.

When $G/H=SO(n)/SO(n-1)$ with $n>2$,
$G_2/SU(3)$ or $Spin(7)/G_2$,
the only homogeneous Finsler metrics
are Riemannian symmetric.

When $G/H=Sp(n)Sp(1)/Sp(n-1)Sp(1)$ with $n>1$ or $Spin(9)/Spin(7)$, we have observed that it
is weakly symmetric and all homogeneous Finsler metrics are
geodesic orbit metrics.

When $G/H=SU(n)/SU(n-1)$ with $n>2$, the connected isometry group $I_o(M,F)$ is isomorphic to $U(n)$ or $SO(2n)$ for any
$SU(n)$-homogeneous Finsler metric $F$.

Summarizing these observations, we see that, to prove the "if" part in the theorem, we only need to prove

\begin{lemma}
 All homogeneous Finsler metrics on $U(n)/U(n-1)$
with $n\geq 1$ and $Sp(n)U(1)/Sp(n-1)U(1)$ with $n>1$ are geodesic orbit metrics.
\end{lemma}

\begin{proof}
To prove this lemma for $G/H=Sp(n)U(1)/Sp(n-1)U(1)$, we consider a reductive decomposition $\mathfrak{g}=\mathfrak{h}+\mathfrak{m}$,
which is orthogonal with respect to a suitably chosen $\mathrm{Ad}(G)$-invariant inner product $|\cdot|_{\mathrm{bi}}^2=\langle\cdot,\cdot\rangle_{\mathrm{bi}}$. We denote the Lie algebra
$\mathfrak{g}=sp(n)\oplus\mathbb{R}v_0$
in which $sp(n)$ is the space of all skew symmetric $n\times n$-matrices with entries in
$\mathbb{H}=\mathbb{R}+\mathbb{R}\mathbf{i}+\mathbb{R}\mathbf{j}
+\mathbb{R}\mathbf{k}$ and $|v_0|_{\mathrm{bi}}=1$. Then $\mathfrak{h}=sp(n-1)\oplus\mathbb{R}$
can be chosen such that the $sp(n-1)$-factor corresponds to the
right-down $(n-1)\times (n-1)$-block, and the $\mathbb{R}$-factor is the diagonal line  $\mathbb{R}(\mathbf{i}E_{1,1},v_0)$, where
$E_{i,j}$ is the $n\times n$-matrix such that its only nonzero entry is 1 in the $i$-th row and $j$-th column.

Then we have the $\mathrm{Ad}(H)$-invariant decomposition
$\mathfrak{m}=\mathfrak{m}_0+\mathfrak{m}_1+\mathfrak{m}_2$.
In this decomposition, $\mathfrak{m}_0$ is linearly spanned by
$e_1=(\mathbf{i}E_{1,1},-v_0)$ with a trivial $\mathrm{Ad}(H)$-action. The summand $\mathfrak{m}_1$ is linearly spanned by $e_2=\mathbf{j}(E_{1,2}+E_{2,1})$ and
$e_3=\mathbf{k}(E_{1,2}+E_{2,1})$, on which the $Sp(n-1)$-factor in $H$ acts trivially, and the $U(1)$-factor
by rotations. The summand $\mathfrak{m}_2$ can be identified
as the space of column vectors $\mathbb{H}^{n-1}=\mathbb{H}\oplus\cdots\oplus\mathbb{H}$,
in which each $\mathbb{H}$-factor is linearly spanned by
\begin{eqnarray*}
e_{4k}= E_{1,k}-E_{k,1},& & e_{4k+1}=\mathbf{i}(E_{1,k}+E_{k,1}),\\
e_{4k+2}=\mathbf{j}(E_{1,k}+E_{k,1})& \mbox{ and }&
e_{4k+3}=\mathbf{k}(E_{1,k}+E_{k,1}),
\end{eqnarray*}
for $1\leq k\leq n-1$. The $\mathrm{Ad}(H)$-actions on $\mathfrak{m}_2$ is the following. The $U(1)$-factor acts by scalar multiplications of $z\in\mathbb{C}$ with $|z|=1$ from the right side, and the $Sp(n-1)$-factor acts by matrix multiplication from the left side. In particular, the natural $Sp(n-1)$-action is transitive on the unit ball in $\mathfrak{m}_2$.

Observing this decomposition of $\mathfrak{m}$
and the $\mathrm{Ad}(H)$-action on each factor, it is not hard to see
there exists a smooth function $f(\cdot,\cdot,\cdot)$, such that for any nonzero vector
$u=\sum_{i=1}^{4n-1}u_ie_i\in\mathfrak{m}$
we have $$F(u)^2=f(u_1,u_2^2+u_3^2,u_4^2+\cdots+u_{4n-1}^2),$$
i.e. any invariant Finsler metric on $Sp(n)U(1)/Sp(n-1)U(1)$
must be of $(\alpha_1,\alpha_2,\beta)$-type. For special $f(\cdot,\cdot,\cdot)$, it may reduce to an $(\alpha_1,\alpha_2)$-metric, $(\alpha,\beta)$-metric or even a Riemannian metric.

By Proposition \ref{prop-1}, to prove $F$ is a geodesic metric,
we only need to prove that the spray vector field $\eta(u)$ is tangent to  $\mathrm{Ad}(H)\cdot u$. Because $\eta(\cdot)$ is
also $\mathrm{Ad}(H)$-invariant, we only need to assume $u=u_1e_1+u_2e_2+u_4e_4$ is a nonzero vector in $\mathfrak{m}$ and prove $\eta(u)$ is contained in $T_u(\mathrm{Ad}(H)\cdot u)=[\mathfrak{h},u]$ which is linearly spanned by $e_i$'s for $i=3$ or $i>4$.

It is easy to show that,
the Hessian matrix $(g_{ij}(u))$ with respect to the linear coordinates
$y=y^ie_i$ satisfies
$g_{ij}(u)=0$ when $i\neq j$, and $i$ or $j$ is not one of
$1$, $2$ and $4$. It is a blocked diagonal matrix, so its
inverse $(g^{ij}(u))$ satisfies similar properties.

Also by direct calculation, we can check that
$$
\langle \eta(u),e_i\rangle_u=\langle u,[e_i,u]_\mathfrak{m}\rangle_u=0,
$$ when $i\notin\{3,5,6\}$. Our previous observation for
$(g^{ij}(u))$ implies that $\eta(u)$ is a linear combination
of $e_3$, $e_5$ and $e_6$, which is tangent to the orbit $\mathrm{Ad}(H)\cdot u$.

This proves all homogeneous Finsler metrics on $Sp(n)U(1)/Sp(n-1)U(1)$ are geodesic orbit metrics.

The proof of this lemma for $U(n)/U(n-1)$ is similar and easier. So we skip the details. Notice that in this case,
any invariant Finsler metric $F$ must be of $(\alpha,\beta)$-type.
\end{proof}

This ends the proof of the "if" part in Theorem \ref{main-thm-1}.

Next we prove the "only if" part in Theorem \ref{main-thm-1}.
The following lemma provides a shortcut.

\begin{lemma}\label{lemma-1}
Let $(G/H,F)$ be a $G$-geodesic orbit space,
and $X$ is a smooth vector field on $G/H$ commuting with all
the Killing vector fields in $\mathfrak{g}=\mathrm{Lie}(G)$,
then $X$ is a Killing vector field of $(G/H,F)$.
\end{lemma}
\begin{proof} We prove this lemma by contradiction, i.e. we assume conversely that $X$ is not a Killing vector field of
$(G/H,F)$.
Let $\rho_t$ be the one-parameter subgroup of diffeomorphisms
generated by $X$. Then there exists some $t_0>0$, $x\in M$
and nonzero vector $y\in T_xM$, such that
$F(x,y)\neq F(\rho_{t_0}(x),(\rho_{t_0})_*y)$.
Notice that $(\rho_{t})_*y\neq y$ for all $t\in\mathbb{R}$ because $\rho_t$'s are diffeomorphisms.
Applying the mean value theorem in calculus, we can find some
$t'\in (0,t_0)$ such that
\begin{equation}
\label{000}
\frac{d}{dt}F(\rho_t(x),(\rho_t)_*y)|_{t=t'}\neq 0.
\end{equation}
Denote $x'=\rho_{t'}$ and $y'=(\rho_{t'})_*y$. Then by Proposition \ref{prop-1}, the $G$-geodesic orbit condition
implies that there exists a Killing vector field $Y\in\mathfrak{g}$, such that $Y(x')=y'$, and $x'$ is
a critical point for the function $f(\cdot)=F(Y(\cdot))$.
Because $X$ commutes with $Y$, we have $Y(\rho_t(x))=(\rho_{t})_* y$ for all $t\in\mathbb{R}$,
and we have
\begin{equation}
\frac{d}{dt}F(\rho_t(x),(\rho_t)_*y)|_{t=t'}
=\frac{d}{dt}F(\rho_t(x),Y(\rho_t(x)))|_{t=t'}= 0.
\end{equation}
which is a contradiction with (\ref{000}).

So $X$ must be a Killing vector field of $(G/H,F)$, which ends
the proof of this lemma.
\end{proof}

Assume conversely that the "only if" claim in Theorem \ref{main-thm-1} is not true, i.e. there exists a homogeneous
Finsler metric on a sphere $M$ such that $I_o(M,F)$
is isomorphic to $Sp(k)$ for some $k\geq 1$. Then $M=S^{4k-1}=Sp(k)/Sp(k-1)$. But $S^{4k-1}$ can also be
presented as $Sp(k)Sp(1)/Sp(k-1)Sp(1)$.
By Lemma \ref{lemma-1}, the $Sp(1)$-factor is also
contained in $I_0(M,F)$, which is
a contradiction to our assumption.

This ends the proof of Theorem \ref{main-thm-1}.

As a by-product, we also get the following corollary.

\begin{corollary}
If a homogeneous Finsler metric $F$ on $M=Sp(n)/Sp(n-1)$ with
$n>0$ is an $Sp(n)$-geodesic orbit metric, then $F$ is
$Sp(n)Sp(1)$-invariant.
\end{corollary}

\section{Killing navigation process}

Navigation process is an important technique in Finsler geometry for constructing new metrics from old ones with some geodesic or curvature properties preserved.

Let $F$ be a Finsler metric and $V$ a smooth vector field on $M$ such that $F(-V(\cdot))<1$ everywhere.
For any $x\in M$ and any nonzero vector $y\in T_xM$, denote $\tilde{y}=y+F(y)V(x)$. Then $\tilde{F}(\tilde{y})=F(y)$ defines new Finsler metric on $M$. We call $(M,\tilde{F})$ the Finsler metric
defined by the {\it navigation datum} $(F,V)$, and this construction the {\it navigation process}. It is not hard to observe that $\tilde{F}(V(\cdot))<1$ everywhere, and
$F$ is also defined by the navigation datum $(\tilde{F},-V)$.

Any Randers metric $F=\alpha+\beta$ can be defined by a unique navigation datum $(\langle\cdot,\cdot\rangle,W)$ where $\langle\cdot,\cdot\rangle$ is a Riemannian metric. The correspondence from $(\langle\cdot,\cdot\rangle,W)$ to $(\alpha,\beta)$ is given by
\begin{eqnarray*}
\alpha=\frac{\sqrt{\lambda \langle y,y\rangle +\langle y,W\rangle^2}}{1-\lambda},\mbox{ and }
\beta=-\frac{\langle y,W\rangle}{1-\lambda},
\end{eqnarray*}
where $\lambda=\langle W,W\rangle<1$.

If the vector field $V$ in the navigation datum $(F,V)$ is a Killing vector field of $(M,h)$, then we call this navigation process a {\it Killing navigation process}.

In later discussions, we will use the following two properties of Killing navigation process proved in \cite{HM2011} and
\cite{HM2007} respectively.

\begin{proposition} \label{killing-navigation-proposition-1}
Let $\tilde{F}$ be the Finsler metric on $M$ defined by the navigation datum $(F,V)$ in which $V$ is
a Killing vector field. Let $c(t)$ be a unit speed geodesic for $F$ and $\rho_t$ the one-parameter subgroup
generated by $V$. Then $\tilde{c}(t)=\rho_t(c(t))$ is a unit speed geodesic for $\tilde{F}$.
\end{proposition}

\begin{proposition}\label{killing-navigation-proposition-2}
Let $\tilde{F}$ be the Finsler metric on $M$ defined by the navigation datum $(F,V)$ in which $V$ is
a Killing vector field. For any $x\in M$, nonzero vector $y\in T_xM$ and tangent plane
$\mathbf{P}=\mathrm{span}\{y,u\}$ with $\langle y,u\rangle_y=0$, denote $\tilde{y}=y+F(y)V(x)$, $\tilde{\mathbf{P}}=\mathrm{span}\{\tilde{y},u\}$ and $\tilde{K}(x,\tilde{y},\tilde{P})$ the flag curvature of the triple $(x,\tilde{y},\tilde{P})$ with respect to the metric $\tilde{F}$, then we have
$$K(x,y,\mathbf{P})=\tilde{K}(x,\tilde{y},\tilde{P}).$$
\end{proposition}

\section{Proof of Theorem \ref{main-thm-2}}

First  we prove the "if" part of Theorem \ref{main-thm-2}.
If $F$ is a homogeneous Randers metric on $M=S^n$ with $n>1$
such that $K\equiv1$, then $F$ corresponds to the navigation
datum $(h,W)$ such that $h$ is the Riemannian metric for
a unit sphere, and $W$ is a Killing vector field of constant
length on $(M,h)$. Obviously the unit sphere $(M,h)$ is a geodesic orbit space, i.e. each geodesic is homogeneous. Then
by Proposition \ref{killing-navigation-proposition-1}, each
geodesic of $(M,F)$ is also homogeneous.

Then we prove the "only if" part of Theorem \ref{main-thm-2}.

Consider a geodesic orbit Finsler sphere $(M,F)=(S^n,F)$ with $n>1$ and $K\equiv 1$. There are two possibilities.

\begin{lemma} Either $(M,F)$ is the standard Riemannian unit sphere or
$I_o(M,F)$ contains a one-dimensional center.
\end{lemma}

The proof is just a summarization of some previous observations
for the list (\ref{list-homo-sphere}) in the proof of Theorem
\ref{main-thm-1}, and a direct application of Theorem \ref{main-thm-1}. Notice that if $M=Sp(m)Sp(1)/Sp(m-1)Sp(1)$ when $n=4m-1$ or $Spin(9)/Spin(7)$ when $n=15$, the homogeneous metric $F$ is reversible. The following theorem in \cite{KM2009} indicates $F$ must be Riemannian.

\begin{theorem}\label{reversible-constant-flag-curvature}
Any reversible Finsler metric with $K\equiv 1$ on a sphere $S^n$ with $n>1$ is Riemannian.
\end{theorem}

To continue the proof of Theorem \ref{main-thm-2},
we may assume $I_o(M,F)$ has a one-dimensional center, i.e.
$G=I_o(M,F)=U(m)$ when $n=2m-1$, or $Sp(m)U(1)$ when
$n=4m-1$.

Any nonzero Killing vector field $V$ in the center of
$\mathfrak{g}=\mathrm{Lie}(G)$ has a constant $F$-length.
Let $c(t)$ be an integration curve of $V$. Then
$c(t)$ is a closed geodesic of constant speed. So is $c(-t)$ because $-V$
is also a Killing vector field of constant length.
Let $\lambda_+$ and $\lambda_-$ be the length of the prime closed geodesics $c(t)$ and $c(-t)$ respectively.
We may assume $\lambda_-\leq\lambda_+$ by choosing $-V$ as $V$ if necessary. Then we have the following lemma.

\begin{lemma}
Keep all notations and assumptions as above, then $\lambda_-\leq 2\pi$.
\end{lemma}

\begin{proof}
On a Finsler sphere $(M,F)=(S^n,F)$ with $n>1$ and $K\equiv1$,
there is a unique Clifford-Wolf translation $\psi$ such that
$d_F(x,\psi(x))=\pi$ \cite{BFIMZ2017}\cite{Sh1997}\cite{Xu2018-preprint1}. The action of $\psi$ preserves each
geodesic, so for any $x=c(0)$, $x'=\psi(x)$
and $x''=\psi^2(x)$ are also points on the geodesic $c(t)$.
We assume conversely that $\lambda_-\geq\lambda_+>2\pi$,
then $x''$ must appear in both connected components of the complement of $\{x,x'\}$ in the geodesic $c(t)$. This is a
contradiction.
\end{proof}

More careful argument can prove $1/\lambda_++1/\lambda_-=1/\pi$, which is not necessary
for later discussion.

For any $\epsilon\in[0,1/F(-V))$, we can use the navigation
datum $(F,\epsilon V)$ to define a new metric $\tilde{F}_\epsilon$.
The metric $\tilde{F}_\epsilon$ is also $G$-homogeneous
because $G$-actions fix $F$ as well as $V$. By Proposition
\ref{killing-navigation-proposition-1},
$\tilde{F}$ is a $G$-geodesic orbit metric. By Proposition
\ref{killing-navigation-proposition-2}, $\tilde{F}_\epsilon$
 has a constant flag curvature $K\equiv1$.

For $\tilde{F}_\epsilon$, $\pm V$ are still Killing vector fields of constant length, so $c(\pm t)$ are closed geodesics
for this new metric. By Proposition \ref{killing-navigation-proposition-1}, the length $\lambda(\epsilon)$
 of the prime closed geodesic $c(-t)$ for the metric $\tilde{F}_\epsilon$ is a smooth monotonous function,
 with
 $$\lambda(0)=\lambda_-,\mbox{ and }
 \lim_{\epsilon\rightarrow1/F(-V)}\lambda(\epsilon)=+\infty.$$
So we can find a suitable $\epsilon'$,
such that $\lambda(\epsilon')=2\pi$. We denote this
$\tilde{F}_\epsilon$ as $\tilde{F}$, and $\tilde{\psi}$
the antipodal map for $\tilde{F}$, then all integration curves of $-V$ provide prime closed geodesics of length $2\pi$, with respect to the metric $\tilde{F}$.
So $(M,\tilde{F})=(S^n,\tilde{F})$ is
a geodesic orbit Finsler sphere with $K\equiv1$ and $\tilde{\psi}^2=\mathrm{id}$.

Finally, we prove $\tilde{F}$ is Riemannian. Then by the navigation process, $F$ is Randers which ends the proof of this theorem. By Theorem \ref{reversible-constant-flag-curvature}, we only need to prove the following lemma.

\begin{lemma}
Let $(M,{F})=(S^n,{F})$ be
a geodesic orbit sphere with $n>1$, $K\equiv1$, and ${\psi}^2=\mathrm{id}$ where ${\psi}$ is the
antipodal map, then ${F}$ is reversible.
\end{lemma}

\begin{proof}
Denote $d_F(\cdot,\cdot)$ the distance function on $M$ defined by ${F}$. We claim

{\bf Claim 1.} $d_F(x_1,x_2)=d_F(x_2,x_1)$ when $d_F(x_1,x_2)=\pi/k$ for some $k\in\mathbb{N}$.

Because ${\psi}^2=\mathrm{id}$, when $k=1$, $\{x_1,x_2\}$
is a $\psi$-orbit, i.e. $d_F(x_2,x_1)=d_F(x_1,x_2)=\pi$. So in later discussion, we may assume $k>1$.

Assume conversely $d_F(x_1,x_2)\neq d_F(x_2,x_1)$. We first consider the possibility that $d_F(x_2,x_1)<d_F(x_1,x_2)=\pi/k$.

There is a unique minimizing geodesic $\gamma$ from $x_1$ to ${\psi}(x_1)$ passing $x_2$. There are a sequel of points
$x_i$'s for $3\leq i\leq k+1$ on $\gamma$, such that $d_F(x_i,x_{i+1})=\pi/k$ for $2\leq i\leq k$ and $x_{k+1}={\psi}(x_1)$.

Using
the geodesic orbit condition, we can find $g_i\in I_o(M,{F})$ for $2\leq i\leq k$ such that
$g_i(x_1)=x_i$ and $g_i(x_2)=x_{i+1}$. Then for $2\leq i\leq k$, we have
$$d_F(x_{i+1},x_i)=d_F(g_i(x_2),g_i(x_1))=d_F(x_2,x_1)<\pi/k,$$
and then
$$d_F({\psi}(x_1),x_1)=d_F(x_{k+1},x_1)\leq\sum_{i=1}^k d_F(x_{i+1},x_i)<\pi,$$
which is a contradiction to
$$d_F({\psi}(x_1),x_1)=
d_F({\psi}^2(x_1),{\psi}(x_1))
=d_F(x_1,x_2)=\pi.$$

So $d_F(x_2,x_1)<d_F(x_1,x_2)=\pi/k$ is impossible.

 Then we consider the possibility that $d_F(x_2,x_1)>d_F(x_1,x_2)=\pi/k$.

 On the minimizing geodesic segment $\gamma$ from $x_1$ to $x_2$, we can find a point $x'$, such that $d_F(x',x_1)=\pi/k$. Obviously $x'\neq x_2$, so $d_F(x_1,x')<d_F(x_1,x_2)=\pi/k$. So we have
 $d_F(x_1,x')<d_F(x',x_1)=\pi/k$, which is impossible by previous
 argument.

This ends the proof of the Claim 1.

Next we claim

{\bf Claim 2.} All geodesics are reversible.

To prove Claim 2, we only need to consider the unique minimizing geodesic segment $\gamma$ from $x_1$ to $x_2$
with $d_F(x_1,x_2)=\pi/2$. By Claim 1, $d_F(x_2,x_1)=\pi/2$,
so there exists another unique minimizing geodesic segment $\gamma'$
from $x_2$ to $x_1$.
On $\gamma$, there is a point
$x'$ with $d_F(x_1,x')=d_F(x',x_2)=\pi/4$. Then by Claim 1,
$d_F(x_2,x')=d_F(x',x_1)=\pi/4$, so $x'\in\gamma'$.
Repeating this argument, we see that the dense subset
$$\{x'|d_F(x_1,x')=\frac{p\pi}{2^q},\mbox{ for some }p,q\in\mathbb{N}\}\subset\gamma$$
is contained in $\gamma'$. So $\gamma\subset\gamma'$, and by
the same argument, $\gamma'\subset\gamma$, i.e. they are the same curve, which proves Claim 2.

The reversibility of ${F}$ follows Claim 1 and Claim 2
immediately, which ends the proof of this lemma.
\end{proof}

\end{document}